\documentclass[review]{elsarticle}
\usepackage{graphicx} % for pictures
\usepackage{natbib} % for bibliography
\usepackage{amsthm}
\usepackage{subcaption}
\usepackage{amsmath,color}
\usepackage{mathtools}
\usepackage{wasysym}
\usepackage{amssymb}
\usepackage[usenames,dvipsnames]{xcolor}

\newtheorem{theorem}{Theorem}[section]

\newtheorem{proposition}[theorem]{Proposition}
\newtheorem{lemma}[theorem]{Lemma}
\newtheorem{example}[theorem]{Example}

\newtheorem{definition}[theorem]{Definition}

\newtheorem{remark}[theorem]{Remark}

\newtheorem{conjecture}[theorem]{Conjecture}

\def\Spec{\mathcal S}
\def\Comp{\mathcal C}
\def\Reac{\mathcal R}
\def\Net{\{ \Spec, \Comp,\Reac \}}

\usepackage{lipsum}
\newcommand\scalemath[2]{\scalebox{#1}{\mbox{\ensuremath{\displaystyle #2}}}}
\usepackage{lineno,hyperref}
\modulolinenumbers[5]
\journal{~}
\date{Apr.\ 21, 2016}

\begin{document}

\begin{frontmatter}

\title{Analyzing Multistationarity in Chemical Reaction Networks using the Determinant Optimization Method}

%% Group authors per affiliation:

%\fnref{label4}}

%\fntext[label3]{I also want to inform about\ldots}
%\fntext[label4]{Small city}

%\ead[url]{author-one-homepage.com}

\author[add2]{Bryan \textsc{F\'{e}lix}}
\address[add2]{Department of Mathematics, University of Texas at Austin, RLM 8.100, 2515 Speedway Stop C1200, Austin, Texas 78712-1202, USA} 
\ead{bryanfelixg@gmail.com}

\author[add3]{Anne \textsc{Shiu}}
\address[add3]{Department of Mathematics, Texas A\&M University, Mailstop 3368, College Station, Texas 77843--3368, USA} 
\ead{annejls@math.tamu.edu}

\author[add4]{Zev \textsc{Woodstock}\corref{cor5}}
\address[add4]{Department of Mathematics and Statistics, James Madison University, Roop Hall 305, MSC 1911, Harrisonburg, Virginia 22807, USA}
\cortext[cor1]{Please send correspondence to ZW.}
\ead{woodstzc@dukes.jmu.edu}

\begin{abstract}
Multistationary chemical reaction networks are of interest to scientists and mathematicians alike. While some criteria for multistationarity exist, obtaining explicit reaction rates and steady states that exhibit multistationarity for a given network---in order to check nondegeneracy or determine stability of the steady states, for instance---is nontrivial.  Nonetheless, we accomplish this task for a certain family of sequestration networks. Additionally, our results allow us to prove the existence of nondegenerate steady states for some of these sequestration networks, thereby resolving a subcase of a conjecture of Joshi and Shiu. Our work relies on the determinant optimization method, developed by Craciun and Feinberg, for asserting that certain networks are multistationary. More precisely, we implement the construction of reaction rates and multiple steady states which appears in the proofs that underlie their method. Furthermore, we describe in detail the steps of this construction so that other researchers can more easily obtain, as we did, multistationary rates and steady states.   

\end{abstract}


\end{frontmatter}
%\linenumbers
\section{Introduction}
Although many dynamical systems arising in applications exhibit bistability, there is no complete characterization of such systems.  Even for the subclass of chemical kinetics systems and even under the assumption of mass-action kinetics, which is the focus of this work, the problem is difficult.  

Here we consider the simpler, yet still challenging, question: 
which chemical reaction networks 
are {\em multistationary}, i.e.\ which have the capacity to exhibit two or more steady-state concentrations with the same reaction rates?  Mathematically, this asks: among certain parametrized families of polynomial systems, which admit multiple positive roots?  Therefore, this is a real algebraic geometry problem, and we do not expect an easy answer in general.

The first partial answers to this question are due to Feinberg, Horn, and Jackson in the 1970s.  Their results in chemical reaction network theory~\cite{FeinOsc,HornJackson72} (specifically, deficiency theory~\cite{FeinDefZeroOne}) can preclude or guarantee
multistationarity for certain classes of networks.  For a survey of these and other methods, see~\cite{Shiu}. 

Our work pertains to two related results: (1) a method for ``lifting'' multiple steady states from small networks to larger ones, and (2) the so-called determinant optimization method for certifying that a given network is multistationary.  

The lifting result, stated informally, is as follows: if a chemical reaction network contains an ``embedded'' network that is multistationary, then the entire reaction network also is multistationary under certain hypotheses~\cite{atoms}. 
Therefore we are interested in cataloguing the multistationary networks which contain no embedded multistationary networks, because all larger multistationary networks contain at least one embedded multistationary subnetwork from the catalogue. 

As a step toward such a catalogue, Joshi and Shiu identified a certain infinite family of chemical reaction networks $\widetilde{K}_{m,n}$ to be of particular interest among all networks that include inflow and outflow reactions~\cite{Shiu}. 
This family is minimal, in that it has no embedded subnetworks (with inflow and outflow reactions) that exhibit multistationarity.
To analyze these networks, Joshi and Shiu used the second method for analyzing multistationarity mentioned above. 

Developed by Craciun and Feinberg, the determinant optimization method can assert that a network is multistationary~\cite{CF,CP}; as such, it is a partial converse to their results on ``injective'' reaction networks which guarantee that a network is {\em not} multistationary.  This topic of injectivity has seen much interest in recent years (see~\cite{CF,BP,signs} and the references therein); however, the determinant optimization method has garnered comparatively little attention.  The only related results that we are aware of are due to Banaji and Pantea~\cite{BP}, Feliu~\cite{feliu}, and M\"uller {\em et al.}~\cite{signs}. 

Using the determinant optimization method, Joshi and Shiu proved that $\widetilde{K}_{m,n}$ is multistationary for all integers $m \geq 2$ and odd integers $n \geq 3$ \cite{Shiu}.  Furthermore, they conjectured that  these networks can exhibit multiple {\em nondegenerate} steady states. (It is {\em not} guaranteed that the determinant optimization method produces nondegenerate steady states; we show this for the first time in Remark~\ref{rem:degen}.)  
The significance of the conjecture is that if it is true, then $\widetilde{K}_{2,n}$ would be the first example of an infinite family of chemical reaction networks with inflow and outflow reactions and at-most-bimolecular reactants and products---that is, minimal with respect to the embedding relation among all such networks which have the capacity to exhibit multiple {\em nondegenerate} steady states. Nondegeneracy is important because results that ``lift'' multiple steady states from embedded subnetworks or other typically smaller networks require the steady states to be nondegenerate; a summary of such results appears in~\cite[\S 4]{Shiu}. Also, because trimolecular reactants/products are rather uncommon in chemistry and $\widetilde{K}_{2,n}$ is at most {\em bimolecular}, this family of networks is of particular interest in chemical applications. 

In the current work, we resolve the conjecture for the case $n=3$ and all $m\geq 2$; in other words, we prove that $\widetilde{K}_{m,3}$ has the capacity to admit multiple nondegenerate steady states for all $m\geq2$ (Theorem~\ref{thm:resolve-conj}). To accomplish this, we need information beyond the mere {\em existence} of multiple steady states; we also need precise values (or at least estimates) for the rates and steady states. By applying the proofs underlying the determinant optimization method in Craciun and Feinberg's work to the networks $\widetilde{K}_{m,n}$, we obtain (via standard methods for analyzing recurrence relations) explicit closed forms for multistationary rates and steady states. Then we use these closed forms to verify that the steady states are nondegenerate for small values of $m$ and $n$. 

Finally, recognizing the usefulness of generating closed forms (or at least estimates\footnote{For general networks, the determinant optimization method need not yield closed forms for the rates and steady states, but one can nonetheless obtain estimates.}) for rates and steady states for any reaction network that satisfies the hypotheses of the determinant optimization method, 
in Section \ref{sec:generate} we outline the steps of the method with enough generality to be used in other contexts. These steps are present in Craciun and Feinberg's work but are spread out over several proofs, so our contribution here is to reorganize the method into a concise procedure. 

An outline of our work is as follows.
Section~\ref{sec:bkrd} introduces chemical systems and the main conjecture.  
Section~\ref{sec:generate} describes the determinant optimization method in detail.
We use this method to resolve some cases of the main conjecture in Sections~\ref{sec:n=3} and~\ref{sec:small-val}.  Finally, a discussion appears in Section~\ref{sec:disc}.

\noindent
{\bf Notation.} We denote the positive real numbers by $\mathbb{R}_+:=\{x \in \mathbb{R} \mid x>0\}$, and the standard inner product in $\mathbb{R}^n$ by $\langle -,-
\rangle$. The $i^{th}$ entry of a vector $x$ is denoted $x_i$.

\section{Background} \label{sec:bkrd}
This section introduces chemical reaction networks, their corresponding mass-action kinetics systems,  
and the main object of our paper: the sequestration network $\widetilde{\textit{K}}_{m,n}$. 

\subsection{Mass-action kinetics systems}

\begin{definition}
A {\em chemical reaction network} $G=\left \{ \mathcal{S,C,R} \right \}$ consists of three finite sets:
\begin{enumerate}
  \item a set of {\em species} $\mathcal{S} = \{ X_1, X_2, \dots, X_s\}$, 
  \item a set $\mathcal{C}$ of {\em complexes}, which are non-negative integer linear combinations of the species, and 
  \item a set $\mathcal{R}\subseteq \mathcal{C} \times \mathcal{C}$ of {\em reactions}. 
\end{enumerate}
\end{definition}

\begin{example}
The following chemical reaction network:
\begin{gather*}
X_1+X_2 \rightarrow X_3 \\
X_2 \rightarrow X_1+X_4~,
\end{gather*}
is entirely defined by:
\begin{enumerate}
    \item the set of species $\mathcal{S}=\{X_1, ~X_2,~ X_3, ~X_4\}$, 
    \item the set of complexes $\mathcal{C}= \{X_1+X_2,~X_3,~X_2,~X_1+X_4 \}$, and
    \item the set of reactions $ \mathcal{R} = \{ (X_1+X_2,X_3 ), (X_2,X_1+X_4 ) \}$.
\end{enumerate}
\end{example}

Any reaction network $G=\left \{ \mathcal{S,C,R} \right \}$ is contained in the {\em{fully open extension network}} $\widetilde{G}$ obtained by including all {\em inflow} and {\em outflow} reactions:
\begin{equation} \label{eq:open}
\widetilde{G}:=\left \{ \mathcal{S,~C\cup S\cup \left \{ \textup{0} \right \},~R\cup }\left \{ X_{i}\leftrightarrow 0 \right \}_{X_{i}\in \mathcal{S}} \right \}~.
\end{equation}
In other words, the fully open extension of any network is obtained by adding the reactions $ X\rightarrow0$ (inflow) and $0 \rightarrow X$ (outflow) for all $X \in \mathcal{S}$.

As all the reactions take place, the concentrations of each of the species will change. We make use of {\em mass-action kinetics} to define a system of ordinary differential equations that describes, for each species, how its concentration changes as a function of time. This ODE system is described by the {\em stoichiometric matrix} $\Gamma$ and the {\em reactant vector} $R(x)$, which is a vector-valued function of the vector of species concentrations {\bf x}.  

\begin{definition}
\label{gamma}
Let $G=\Net$ be a network, and let $\{ y_1 \to y_1', y_2 \to y_2', \dots , y_{|\Reac|} \to y_{|\Reac|}' \}$ be an ordering of the reactions.
\begin{enumerate}
\item
The {\em reaction vector} of the reaction $y_i \to y_i'$ is the vector $y_i' - y_i$, viewed in $\mathbb{R}^{|\Spec|}$. Note that $y_i \to y_i'$ is a slight abuse of notation, used to denote the reaction $y_i \cdot X \to y_i' \cdot X$ where $X$ is the vector of all species. Explicitly, the vectors $y_i$ and $y_i'$ only contain species coefficients.
\item The {\em stoichiometric matrix} of $G$ is the $\left | \mathcal{S} \right | \times \left | \mathcal{R}\right |$ matrix $\Gamma$ whose $k^{\rm th}$ column is the reaction vector of $y_k \to y_k'$.

\item
The {\em reactant vector}  $R(x)$
is the vector of length $|\Reac|$ whose $k^{th}$ entry 
is the (monomial) product:
\begin{gather*}
r_{k} x_1^{(y_k)_1} x_2^{(y_k)_2} \cdots x_{\left | \mathcal S \right |}^{(y_k)_{\left | \mathcal S \right |}}~,
\end{gather*}
where $r_k \in \mathbb{R}_+$ is the {\em reaction rate} of the $k^{th}$ reaction.
\end{enumerate}
\end{definition}

\begin{definition} \label{eq:mass-action}
The {\em mass-action kinetics system} of a network $G=\Net$ and a vector of reaction rates $(r_k)\in \mathbb{R}_+^{|\Reac|}$ is defined by the following system of ordinary differential equations:
\begin{equation}
\label{M-A System}
\frac{d\mathbf{x}}{dt} = \Gamma \cdot R(\mathbf{x})~.
\end{equation}
\end{definition}

\begin{example}
For the following network:
\begin{gather*}
A+2B \overset{r}{\rightarrow} 2A~,
\end{gather*}

$\Gamma = \begin{bmatrix}
1\\ 
-2
\end{bmatrix}$
and $R(x)=\left( r x_A x_B^{2} \right)$, so 
the mass-action kinetics system~(\ref{M-A System}) is: 
\begin{gather*}
\begin{bmatrix}
\frac{dx_A}{dt} \\
\frac{dx_B}{dt}
\end{bmatrix}
~=~ \Gamma \cdot R(x) ~=~\begin{bmatrix}
r x_A x_B^{2}\\ 
-2r x_A x_B^{2}
\end{bmatrix}~.
\end{gather*}
\end{example}
An important characteristic of mass-action kinetics systems is that they may or may not have the capacity to admit (positive) steady states: 

\begin{definition}
A {\em positive steady state} is a vector $\mathbf{x}^* \in \mathbb{R}^{|\mathcal{S}|}_{>0}$ such that $\Gamma \cdot R(\mathbf{x}^*)=0$.
A steady state $\mathbf{x^{*}}$ 
is {\em nondegenerate} if  ${\rm Im}(df(\mathbf{x^{*}})|_{{\rm Im}(\Gamma)}) = {\rm Im}(\Gamma)$, where $df(\mathbf{x^{*}})$ denotes the Jacobian matrix of the mass-action kinetics system at $\mathbf{x^{*}}$.
\end{definition}

% DEFINITION OF MULTISTATIONARY
\begin{definition}
A network that includes all inflow and outflow reactions is {\em multistationary}\footnote{The focus of this work is on certain networks $\widetilde{K}_{mn,}$ that include all flow reactions.  For networks that do not include all flow reactions, the definition of multistationary must incorporate the conservation relations in the network, if any.} if there exist two distinct concentration vectors $\mathbf{x^*, x^\#}$ and positive reaction rates such that $\Gamma \cdot R(\mathbf{x^*}) = \Gamma \cdot R(\mathbf{x^\#}) = 0$.
\end{definition}

\subsection{The sequestration network $K_{m,n}$}
The main object of our paper is the {\em fully open extension} of the network $\textit{K}_{m,n}$:

\begin{definition} \label{def:seq}
For positive integers $n\geq 2$ and $m\geq 2$, the {\em sequestration network} $\mathit{K}_{m,n}$ is: 
\begin{align} \label{eq:seq}
X_{1}+X_{2}&\rightarrow 0\\
X_{2}+X_{3}&\rightarrow 0 \nonumber \\
& ~~ \vdots \nonumber \\
X_{n-1}+X_{n}&\rightarrow 0 \nonumber \\
X_{1}&\rightarrow mX_{n}~. \nonumber
\end{align}
$\widetilde{K}_{m,n}$ is the fully open extension of $\mathit{K}_{m,n}$, obtained by adjoining all inflow and outflow reactions, as in~\eqref{eq:open}. 
\end{definition}

Schlosser and Feinberg analyzed variations of $\widetilde{K}_{2,n}$~\cite[Table 1]{SF}, as did Craciun and Feinberg~\cite[Table 1.1]{CF}.  Joshi and Shiu introduced the version of the sequestration networks in Definition~\ref{def:seq}, and proved that some of them are multistationary:

\begin{proposition}{\cite[Lemma~6.9]{Shiu}}
\label{prop:mss}
For positive integers $m \geq 2$ and $n \geq 3$, if $n$ is odd, then $\widetilde{K}_{m,n}$ admits multiple positive steady states. 
\end{proposition}
For $m=1$ or $n$ even, the network $\widetilde{K}_{m,n}$ is ``injective'' and therefore not multistationary~\cite[\S 6]{Shiu}.  
Joshi and Shiu conjectured that Proposition~\ref{prop:mss} extends as follows:
\begin{conjecture}
\label{Conjecture}
For positive integers $m \geq 2$ and $n \geq 3$, if $n$ is odd, then $\widetilde{K}_{m,n}$ admits multiple {\em nondegenerate} steady states.
\end{conjecture}
To resolve Conjecture~\ref{Conjecture}, we must show that $Im(df(\mathbf{x^*})) = Im(\Gamma)$ for two distinct positive steady states $\mathbf{x^*}$. We will see in~\eqref{eq:stoic-mat} below that $\Gamma$ is full rank, so we need only show that $\det(df(\mathbf{x}^*)) \neq 0$ for two positive steady states $\mathbf{x^*}$. 

\begin{remark}
As mentioned in the introduction, networks in which all reactants and products are at most {\em bimolecular}---that is, each complex has the form $0$, $X$, $X+Y$, or $2X$---are the norm in chemistry. 
This is the case for the networks $\widetilde{K}_{2,n}$, so that the $n^{th}$ internal reaction is $X_1 \rightarrow 2 X_n$.
\end{remark}

We end this section by displaying the matrices that define the mass-action kinetics system (\ref{M-A System}) defined by $\widetilde{K}_{m,n}$. We order the reactions as follows: first, we enumerate the $n$ {\em internal} (or {\em true}) reactions listed in~\eqref{eq:seq} (so, the first reaction is $X_1 + X_2 \to 0$, and so on), 
next are the $n$ outflow reactions (so, the ($n+1$)st reaction is $X_1 \to 0$, and so on), and then we have the $n$ inflow reactions (so, the ($2n+1$)st reaction is $0 \to X_1$, and so on).  We will refer to the sets of internal (true), outflow, and inflow reactions as $\mathcal{R}_T$, $\mathcal{R}_O$, and $\mathcal{R}_I$, respectively. 

The stoichiometric matrix for $\widetilde{K}_{m,n}$ is:
\begin{equation} \label{eq:stoic-mat}
\Gamma
=
\left[
\begin{array}{cccccc|c|c}
    -1 & 0  &  0 & \dots  & 0 & -1 & &\\
    -1 & -1 &  0 & \dots  & 0 & 0 & &\\
     0 & -1 & -1 &  \ddots & \vdots & \vdots & -I^n & I^n \\
     0 & 0  &  -1 & \ddots & \vdots & \vdots & &\\
    \vdots & \vdots & \ddots & \ddots & -1 & 0 & &\\
    0 & 0 & 0 & \dots  & -1 & m & &
\end{array}
\right]~,
\end{equation}
where $I^n$ is the $n \times n$ identity matrix.  
The reactant vector is: % given by,

\begin{equation*}
R(\mathbf{x}) 
=
\begin{bmatrix}
r_1x_1x_2 \\
r_2x_2x_3 \\
\vdots \\
r_{n-1}x_{n-1}x_n \\
\hline 
r_nx_1 \\
\hline
r_{n+1}x_1 \\
r_{n+2}x_2 \\
\vdots \\
r_{2n}x_n \\
\hline
r_{2n+1} \\
\vdots \\
r_{3n}
\end{bmatrix}~,
\end{equation*}
where the $r_i \in \mathbb{R}_{+}$ are the reaction rates and each $x_i \in \mathbb{R}_{+}$ is the concentration of each species $X_i$. 
%----------------------
%HERE ARE THE ODEs
%----------------------
The mass-action ODEs~\eqref{eq:mass-action} are:
\begin{align*}
%\label{ODEs}
%\begin{matrix}
\dot{x}_1~=& -r_1x_1x_2 - r_nx_1 - r_{n+1}x_1 + r_{2n+1} \\
\dot{x}_i~=& -r_{i-1}x_{i-1}x_i - r_ix_ix_{i+1} - r_{n+i}x_i + r_{2n+i} \quad \quad {\rm for~} 2 \leq i \leq n-1 \\ %\ i \in \{2, 3, ..., n-1\} \\
\dot{x}_n~=& -r_{n-1}x_{n-1}x_n + mr_nx_1 - r_{2n}x_n + r_{3n}~.
%\end{matrix}
\end{align*}
%----------------------
% JACOBIAN MATRIX
%----------------------
Thus the Jacobian matrix, $df(\mathbf{x})$, is the following $(n \times n)$-matrix
\small{
\begin{equation} \label{eq:jac}
\scalemath{0.638}{
\begin{bmatrix}
-r_1x_2 - r_n - r_{n+1} & -r_1x_1  & 0 &\dots & 0 &0\\
-r_1x_2 & -r_1x_1 - r_2x_3 - r_{n+2} & -r_2x_2  & \dots & \vdots &\vdots\\
0 & -r_2x_3 & -r_2x_2 - r_3x_4 - r_{n+3} & \ddots & 0 &0\\
\vdots &0 &-r_3x_4 &\ddots &-r_{n-2}x_{n-2} &0\\
0 & \vdots & \vdots & \ddots    & -r_{n-2}x_{n-2}-r_{n-1}x_{n}-r_{2n-1} &-r_{n-1}x_{n-1}\\
mr_n & 0 & 0 & \dots & -r_{n-1}x_n & -r_{n-1}x_{n-1} - r_{2n}
\end{bmatrix}
}~.
\end{equation}

%----------------------
% SECTION: the det. opt. method
%----------------------
\section{Constructing multiple steady states via the determinant optimization method}
\label{sec:generate}
The {\em determinant optimization method}\footnote{Related techniques for establishing multistationarity appear in work of Banaji and Pantea~\cite[\S 4]{BP}, Feliu~\cite[\S 2]{feliu}, and M\"uller {\em et al.}~\cite[\S 3.2]{signs}.} was developed by Craciun and Feinberg to show that certain chemical reaction networks are multistationary~\cite{CF}. More precisely, the method guarantees that some networks (such as those that satisfy the `Input' conditions below) are necessarily multistationary.  For instance, Joshi and Shiu showed that the networks $\widetilde{K}_{m,n}$ (for $m \geq 2$ and odd $n \geq 3$) satisfy the `Input' conditions, and thus concluded these networks are multistationary (Proposition~\ref{prop:mss}). 

In fact, the determinant optimization method also applies to some networks that do not satisfy the `Input' conditions.  To determine if this is the case for a given network, one must check whether a certain optimization problem has a solution (see Remark~\ref{rmk:skip-ahead-steps}).  If so, then the method guarantees that the network is multistationary.

However, in many applications, it is useful not only to know that a network is multistationary but also to have explicit steady-state concentrations and reaction rates that are witnesses to multistationarity.  For instance, here we would like to determine whether the steady states are degenerate, whereas in other settings one might like to perform stability analysis.

Fortunately, the proofs in~\cite[\S 4]{CF} that underlie the determinant optimization method are constructive, up to one use of the Intermediate Value Theorem, so one can generate or at least approximate steady states and rates. 
This section describes the step-by-step procedure to do this; following our steps is easier than (although equivalent to) ``backtracking'' through the proofs.  That is, our contribution here is to re-package the determinant optimization method into a constructive algorithm.  We will see that for some networks, such as $\widetilde{K}_{m,n}$, the method constructs closed forms for the steady states and rates.

%----------------------------
% Det. opt. method
%----------------------------
\vskip .1in
\noindent
{\sc Determinant optimization method (constructive version)} \\
{\bf Input:} Any chemical reaction network $G$ with $n = |\Spec|$ species that contains all $n$ inflow reactions such that 
%that satisfies the following hypotheses for 
there exist $n$ reactions $y_1 \to y_1'$, $y_2 \to y_2'$, \ldots ,
$y_n \to y_n'$ among the internal (true) and outflow reactions $\Reac_T \cup \Reac_O$ of $G$ for which
    \begin{enumerate}[(I)]
    \item $\det(y_1, y_2, ..., y_n) \cdot \det((y_1 - y_1'), (y_2 - y_2'), ..., (y_n - y_n')) < 0$, and 
    \item there exists a vector $\tilde{\eta} \in \mathbb{R}^{n}_+$ such that $\Sigma_{i = 1}^{n} \tilde{\eta}_i (y_i - y_i') \in \mathbb{R}_+^{n} = \mathbb{R}_+^{|S|}$. 
    \end{enumerate}

\noindent
{\bf Output:} A certificate of multistationarity of $G$: (approximations of) a positive reaction rate vector $(r_{y \to y'}) \in \mathbb{R}^{|\Reac|}_{+}$ and two positive concentration vectors $x^*$ and $x^{\#}$ which are both steady states of the mass-action system defined by $G$ and  $(r_{y \to y'})$. \\
{\bf Steps:} Described below.

\vskip 0.1in
%----------------------------

With an eye toward resolving Conjecture~\ref{Conjecture}, $\widetilde{K}_{m,n}$ 
will be our ongoing example. 

\begin{example}
For $\widetilde{K}_{m,n}$ (with $m \geq 2$ and $n \geq 3$ odd), hypothesis~(II) is satisfied by the vector $\tilde{\eta}= (1,1,\dots, 1, m+1, 1)$~\cite[Lemma 6.9]{Shiu}.  Hypothesis (I) was proven in \cite[Lemma 6.7]{Shiu}, where the $n$ reactions are precisely the $n$ internal (true) reactions~\eqref{eq:seq}.  Conveniently, these are the reactions labeled $y_i \to y_i'$ of the sequestration network, for $1 \leq i \leq n$, so our notation for the first $n$ reactions---as well as the use of $n$ for the number of species---matches that of the determinant optimization method.
 
\end{example}

The steps below involve a certain linear transformation $T_\eta$; specifically, for $\eta \in \mathbb{R}^{\mathcal{R}_T \cup \mathcal{R}_O}$, the linear transformation $T_\eta: \mathbb{R}^{|\mathcal{S}|} \to \mathbb{R}^{|\mathcal{S}|}$ is defined by: 
\begin{equation}
\label{transformation}
T_\eta (\delta) ~=~ \sum_{y \rightarrow y' \in \mathcal{R}_T \cup \mathcal{R}_O} \eta_{y \rightarrow y'} (y \cdot \delta) (y - y')~.
\end{equation}
Equivalently, the matrix representation of $T_{\eta}$ is $d(-f)(1,1,\dots,1)$ where the rates are given by $r_i = \eta_i$.  In other words, this matrix is the Jacobian matrix of the mass-action system~\eqref{eq:mass-action} defined by the internal and outflow reaction rates $\eta$ (and any choice of inflows: they do not appear in the Jacobian matrix) at the concentration vector $(1,1,\dots, 1)$.

\begin{example}
For $\widetilde{K}_{m,n}$, the matrix representation of $T_\eta$ is:
\begin{equation} \label{eq:T}
\scalemath{0.89}{
\begin{bmatrix}
\eta_1 + \eta_n + \eta_{n+1} & \eta_1 & 0 & \cdots & 0\\
\eta_1 & \eta_1 + \eta_2 + \eta_{n+2} & \eta_2 & \cdots & 0\\ 
0 & \eta_2 & \eta_2 + \eta_3 + \eta_{n+3} & \eta_3 &  0  \\ 
\vdots & 0 & \ddots & \ddots & \vdots\\
0 & \vdots & \cdots & \eta_{n-2} + \eta_{n-1} + \eta_{2n-1} & \eta_{n-1} \\
-m \eta_n & 0 & \cdots & \eta_{n-1}& \eta_{n-1} + \eta_{2n}~.
\end{bmatrix}}
\end{equation}
From the Jacobian matrix~\eqref{eq:jac}, it is clear that this matrix~\eqref{eq:T} equals $d(-f)(1,1,\dots,1)$, where the reaction rates are given by $r_i = \eta_i$.
\end{example}

%--------------
% FIRST STEP
%--------------
\noindent
{\bf The first step} is to construct a (strictly positive) vector $\eta^- \in \mathbb{R}_+^{|\Reac_T \cup \Reac_O|}$, indexed by all internal (true) and outflow reactions, such that
    \begin{enumerate}[(I)]
    \item $\det(T_{\eta^-}) <0$, and 
    \item $\sum\limits_{y \to y' \in \Reac_T \cup \Reac_O} \eta^-_{y \to y'} (y - y') \in \mathbb{R}_+^{|\Spec|}.$ %, where $(y_i - y_i')$ is the negation of the $i^{th}$ column in $\Gamma(x)$. 
    \end{enumerate}
Craciun and Feinberg proved that these conditions (I) and (II) are satisfied by a vector $\eta^-$ of the following form:
\begin{equation*}
\eta ^{-}_{y \to y'}=
\begin{cases}
\lambda \tilde{\eta}_{y \to y'} & \text{ if~ } y \to y' \in \{y_i \to y_i' \mid i \in[n] \} \\ 
\epsilon & \text{ else },
\end{cases}
\end{equation*}
where $\lambda$ is sufficiently large and $\epsilon$ is sufficiently small~\cite[proof of Theorem 4.2]{CF}.

\begin{example}
For $\widetilde{K}_{m,n}$ (with $m \geq 2$ and $n \geq 3$ odd), we define $\eta^-$ as follows:
\begin{equation*}
\eta ^{-}_{i}=
\begin{cases}
\lambda & \text{ if }1 \leq i \leq n-2 \text{ or } i=n\\ 
(m+1)\lambda & \text{ if } i= n-1 \\ 
\epsilon & \text{ if } n+1 \leq i \leq 2n~.
\end{cases}
\end{equation*}
\end{example}

\begin{remark}[Stronger versions of the determinant optimization method] \label{rmk:skip-ahead-steps}
This section describes the
simplest version of the determinant optimization method.  In fact, even if a network does not satisfy the hypotheses in the input given above, the method can still apply: \cite[Remark 4.1]{CF} describes, in this setting, how to implement the above first step, i.e.\ how to test whether a suitable $\eta^-$ exists, as an optimization problem.  Specifically, this is a polynomial optimization problem with linear constraints over a compact set.  Therefore, one can use any applicable optimization method.
Additionally, see Remark~\ref{rmk:interpret} for how one can begin the algorithm at the second step.
\end{remark}

%--------------
% SECOND STEP
%--------------
\noindent
{\bf The second step} is to construct
a (strictly positive) vector $\eta^0 \in \mathbb{R}_+^{|\Reac_T \cup \Reac_O|}$ for which:
    \begin{enumerate}[(I$'$)]
    \item $\det(T_{\eta^0}) =0$, and 
    \item $\sum\limits_{y \to y' \in \Reac_T \cup \Reac_O} \eta^0_{y \to y'} (y - y') \in \mathbb{R}_+^{|\Spec|}.$     \end{enumerate}
Craciun and Feinberg proved that this can be accomplished as follows~\cite[proof of Theorem 4.1]{CF}. 
First, construct an $\eta^+ \in \mathbb{R}^{|\mathcal{R}_T \cup \mathcal{R_O|}}_+$ such that $\det(T_{\eta^+}) > 0$;  
do this by assigning a large value to outflow reactions and a small value to internal reactions:
\begin{equation*}
\eta ^{+}_{y \to y'}=
\begin{cases}
\lambda^+ & \text{ if~ } y \to y' \in \mathcal{R}_O \\ \epsilon^+ & \text{ if~ } y \to y' \in \mathcal{R}_T~,
\end{cases}
\end{equation*}
where $\lambda^+>0$ is large and $\epsilon^+ >0$ is small.

Then, by interpolating between this vector $\eta^+$ and the vector $\eta^-$ from the previous step, the Intermediate Value Theorem (plus the fact that the set of vectors satisfying condition (II) is convex) guarantees the existence of an $\eta^0$ with the required properties. Moreover, such a suitable $\eta^0$ can be numerically approximated, and, with careful tracking of error, one can use this approximation to generate steady states and concentrations in the following steps.  

\begin{remark}[Interpretation of the second step and subsequent steps] \label{rmk:interpret}
What the second step does is to find reaction rates (given by $\eta_0$ for the internal and outflow rates, and the vector in $(II')$ for the inflow rates) at which the concentration vector $(1,1,\dots,1)$ is a degenerate steady state: degeneracy is by $(I')$, and being a steady state comes from $(II')$.  

Equivalently, if $(\widetilde{r}_{y \to y'})$ is any positive vector of reaction rates at which some concentration vector $\widetilde{c}$ is a degenerate positive steady state, then the vector $\eta \in \mathbb{R}_+^{|\Reac_T \cup \Reac_O|}$ defined coordinate-wise by $\eta_{y \to y'}= \widetilde{r}_{y \to y'} \widetilde{c}^y$ satisfies the second step.  So, a reader who has already found a degenerate positive steady state of their system could start the determinant optimization method at the second step.  In other words, one could begin applying the method by immediately searching for a suitable $\eta^0$ (without first generating $\eta^-$ and $\eta^+$).  One strategy for doing this is described in Remark~\ref{rmk:strategy}, which we employ for $\widetilde{K}_{m,n}$ beginning in Example~\ref{ex:eta-0}.

In the next steps, the determinant optimization method constructs a certain vector $\delta$ so that $|\delta|$ is a suitable bifurcation parameter: for $|\delta|$ small but positive, the degenerate steady state breaks into two nondegenerate steady states.
\end{remark}

\begin{remark} \label{rmk:strategy}
Here is one strategy for constructing a suitable $\eta^0$ (without using $\eta^-$ and $\eta^+$).  First, identify (if possible) an $\eta \in \mathbb{R}^{|S|}_+$ and a reaction $y_i \to y_i'$ among the internal (true) and outflow reactions such that:
\begin{enumerate}[(a)]
\item $\sum\limits_{y \to y' \in \left( \Reac_T \cup \Reac_O \right) \setminus \{y_i \to y_i'\} } \eta^0_{y \to y'} (y - y') \in \mathbb{R}_+^{|\Spec|}$, and
\item $y_i - y_i' \in \mathbb{R}_{\geq 0}^{|S|}$ (this holds, for instance, if $y_i \to y_i'$ is an outflow reaction).
\end{enumerate}
One could see whether $\eta^+$ or $\eta^-$ might work (we use $\eta^-$ in Example~\ref{ex:eta-0} below).  Then, define $\eta$ as follows: let the entry $\eta^0_i$ (corresponding to the same $i^{th}$ reaction) be free, and fix $\eta^0_j = \eta^-_j$ for $j \neq i$.  Then, solve the (univariate polynomial) equation $\det(T_{\eta^0})=0$. 
If there is a positive solution (for $\eta^0_i$), then the resulting vector $\eta^0$ is positive, and $(I')$ holds by construction. Furthermore, $(II')$ holds because the sum in $(II')$ is precisely the sum of a positive vector (namely, the sum in (a)) and a non-negative vector (namely, $\eta^0_i (y_i - y_i')$).
%If there is a positive solution (for $\eta^0_i$), then the resulting vector $\eta^0$ is positive, and $(I')$ and $(II')$ hold automatically, by construction and by the setup (the sum in $(II')$ is the sum of that in (a) and $\eta^0_i (y_i - y_i')$, which is non-negative), respectively. 
However, $\eta^0_i$ is not guaranteed to be positive, so this strategy may fail.
\end{remark}

\begin{example} \label{ex:eta-0}
For $\widetilde{K}_{m,n}$ (with $m \geq 2$ and $n=3,5,7,9,11$), 
the following choice of $\eta^0$ satisfies the requirements of the second step:
\begin{equation}
\eta ^{0}_{i}=
\begin{cases}
    \lambda & \text{ if }1 \leq i \leq n-2 \text{ or } i=n  \\ 
    (m+1)\lambda & \text{ if } i= n-1  \\ 
    \epsilon & \text{ if } n+1 \leq i \leq 2n-1  \\
    \frac{(m+1)(m\lambda^n + \lambda^2 (m+1) \daleth_{n-2})}{(\lambda(m+2) + \epsilon)\daleth_{n-2} - \lambda^2\daleth_{n-3}} - \lambda(m+1)     & \text{ if } i=2n    ~, \label{eq:eta-2n}
\end{cases}
\end{equation}
where $\lambda>0$ and $\epsilon>0$ are such that $\eta^0_{2n}$ is positive\footnote{We checked that such $\lambda, \epsilon$ exist for $n=3,5,7,9,11$ (and $m \geq 2$), and we furthermore conjecture that for larger $n$, choosing $\lambda$ sufficiently large and $\epsilon$ sufficiently small will suffice.} (thus, all coordinates of $\eta^0$ are positive), and $\daleth_i$ is $i^{th}$ principal minor of the matrix representation of $T_{\eta^0}$ displayed below in~\eqref{T0}, i.e. $\daleth_i$ is the determinant of the $i \times i$ (tridiagonal) upper-left submatrix of~\eqref{T0} (also, $\daleth_0:=1$).

To show that this choice of $\eta^0$ satisfies the two conditions of the second step, we first note that $(II')$ is straightforward to verify.

Satisfying $(I')$ only requires $\eta^0$ to satisfy one (determinantal) equation, so allowing one free variable is sufficient. We choose $\eta^0_{2n}$ as this free variable, and we will recover the formula in~\eqref{eq:eta-2n}.  Namely, we let all other coordinates $\eta_i^0$ have the form given above (for $1 \leq i \leq 2n-1$), and then, recalling~\eqref{eq:T}, the matrix representation of $T_{\eta^0}$ is:
\begin{equation}
\label{T0}
%T_{\eta^0} = 
\begin{bmatrix}
2 \lambda + \epsilon & \lambda & 0 & &  \cdots & 0\\
\lambda & 2\lambda + \epsilon & \lambda & 0 & \cdots  & 0\\ 
0 & \lambda & \ddots  & \ddots \ \ \ &  &  \vdots  \\ 
\vdots & 0 & \ddots &  & & 0 \\
0 & \vdots &  & & \lambda + \lambda(m+1) + \epsilon & \lambda(m+1) \\
-m \lambda & 0 & \cdots & 0 & \lambda(m+1)& \lambda(m+1) + \eta_{2n}^0
\end{bmatrix}.
\end{equation}

Expanding (\ref{T0}) along the bottom row, we obtain the determinant of $T_{\eta^{0}}$:
\begin{equation}
\label{detteta}
\det(T_{\eta^0}) = (-1)^n m(m+1) \lambda^n - \lambda^2(m+1)^2 \daleth_{n-2} + (\lambda (m+1) + \eta^0_{2n}) \daleth_{n-1}~,
\end{equation}
where we recall that
 $\daleth_i$ is $i^{th}$ principal minor of $T_{\eta^0}$.

The determinant of tridiagonal matrices can be solved recursively \cite{tridiagonal}. For our matrix, it is easy to verify the following recursion for $i \leq n-2$, which is independent of $m$:
\begin{align*}
\daleth_{i+2} &= (2\lambda + \epsilon) \daleth_{i+1} - \lambda^2\daleth_i~,
\end{align*}
with initial values $\daleth_0 := 1$ and $\daleth_1 = 2\lambda + \epsilon$.
%is true for entries of $T_\eta$ independent of $m$, i.e. this works for $i \leq (n-2)$. 
Notice that $\daleth_{n-1}$ must be treated separately, because the $(n-1)^{st}$ row contains $\eta^0_{n-1}$, which is a function of $m$. Using standard methods we can get the generating function of the recurrence. 
\begin{equation*}
\daleth_i = \frac{1}{2^{i+1} c_1} * \left( -c_2(c_2 - c_1)^i + c_1(c_2 - c_1)^i + c_2(c_1 + c_2)^i + c_1(c_1 + c_2)^i \right)~,
\end{equation*}
where $c_1 = (\epsilon)^{\frac{1}{2}} (\epsilon + 4 \lambda)^{\frac{1}{2}}$ and $c_2 = \epsilon + 2 \lambda$. Notice here that $\daleth_i$ is always positive for sufficiently small $\epsilon$. A formula for $\daleth_{n-1}$ is given using the formula for tridiagonal matrices \cite{tridiagonal}:
\begin{equation*}
\daleth_{n-1} = \daleth_{n-2}(\lambda (m+2) + \epsilon) - \lambda^2 \daleth_{n-3}~.
\end{equation*}
With this recurrence solved, we set equation (\ref{detteta}) to zero and then derive the explicit function for $\eta^0_{2n}$ from (\ref{detteta}) in terms of $m, \lambda,$ and $\epsilon$; this is the formula in~\eqref{eq:eta-2n}.  

\end{example}

%--------------
% THIRD STEP
%--------------
\noindent
{\bf The third step} is to construct a nonzero vector $\delta \in \mathbb{R}^{|\Spec|}$ in the nullspace of $T_{\eta^0}$, i.e. such that 
$T_{\eta^0} \cdot \delta = 0.$  (Such a vector $\delta$ exists because 
%the nullspace of $T_{\eta^0}$ is non trivial, since 
$\det(T_{\eta^0})=0$.) 

\begin{example} \label{ex:delta} 
For our example $\widetilde{K}_{m,n}$ (with $m \geq 2$ and $n \geq 3$ odd), we claim that the vector $\delta$ whose coordinates are defined as follows is in the nullspace of (\ref{T0}): %It suffices to show that the rank of $T_{\eta^0}$ is equal to $n-1$. Then, the result follows from the rank-nullity theorem. 
%We proceed by showing that the first $n-1$ rows of $T_{\eta^0}$ are linearly independent. Assume, to the contrary, that they are not linearly independent and there is a non-trivial linear combination of the first $n-1$ rows of $T_{\eta^0}$ that adds to 0.
%Note that  the first $n-1$ rows only the last one contains an entry in the last column, namely $\eta^0_{n-1}$. Since $\eta^0_{n-1}\neq 0$ the corresponding scalar of the $n-1$ row should be 0, thus, annihilating the entire row. In the same manner, the corresponding scalar for the $n-2$ row would be equal to 0. Repeating the process shows that the only linear combination that adds to 0 is the trivial one, thus, arriving at a contradiction.
%\end{proof}
\begin{equation} \label{eq:delta}
\delta_k= 
\begin{cases}
    \delta_1 & \text{ if }k=1  \\ 
    \frac{-(2\lambda+\epsilon)}{\lambda}\delta_{k-1}-\delta_{k-2} & \text{ if } 2 \leq k \leq n-1\\
    \frac{-(\lambda(m+2)+\epsilon)}{\lambda(m+1)}\delta_{n-1}-\frac{1}{m+1}\delta_{n-2}  & \text{ if } k=n~,
\end{cases}
\end{equation}
where we introduce $\delta _0=0$ for convenience in solving the recurrence relation, and $\delta_1\neq 0$ is our free variable.
Note that the last coordinate, $\delta_n$, has a different formula, because the $(n-1)^{st}$ row in (\ref{T0}) used to define $\delta_n$ contains terms dependent on $m$ that do not satisfy the recurrence.

To see that $\delta$ is a nonzero vector in the nullspace of $T_{\eta^0}$, notice that the conditions on $\delta$ that state that its inner product with each of the the first $n-1$ rows of $T_{\eta^0}$ coincide precisely with the $n-1$ recurrences in the definition of $\delta$~\eqref{eq:delta}.  We claim that the last row of $T_{\eta}$ is linearly dependent on the other rows, and from this we will conclude that the last row of $T_{\eta_0}$ automatically has zero inner product with $\delta$.  To see this, we recall that $\det T_{\eta^0}=0$ by construction, so we need only show that the first $n-1$ rows of $T_{\eta^0}$ are linearly independent. Assume, to the contrary, that there is a non-trivial linear combination of the first $n-1^{st}$ rows of $T_{\eta^0}$ that adds to 0.  Note that in the first $n-1$ rows only the last one contains an entry in the last column, namely $\eta^0_{n-1}$. Since $\eta^0_{n-1}\neq 0$, the corresponding scalar of the $n-1^{st}$ row should be 0, thus annihilating the entire row. In the same manner, the corresponding scalar for the $n-2^{nd}$ row would be equal to 0. Repeating the process shows that the only linear combination that adds to 0 is the trivial one, thus, arriving at a contradiction.  So, $\delta$ is in the nullspace of $T_{\eta^0}$.

Again, by using standard techniques for analyzing recurrences, we find the generating function for each of the first $n-1$ entries of $\delta$:
\begin{equation*}
%\label{enddelta}
\delta_{k}=\delta _{1} \lambda \cdot \frac{(\sqrt{4\lambda\epsilon+\epsilon^{2}}-(2\lambda+\epsilon))^{k}-(-\sqrt{4\lambda\epsilon+\epsilon^{2}}-(2\lambda+\epsilon))^{k}}{2^{k} \lambda^{k}\sqrt{4 \lambda \epsilon+\epsilon^{2}}}~,
\end{equation*}
for $1 \leq k \leq n-1$.
\end{example}

%--------------
% LAST STEP
%--------------
\noindent
{\bf The final step} is to use the vectors $\eta^0$ and $\delta$ (or, as we will see, a sufficiently scaled version of $\delta$) from the previous two steps to construct a certificate of multistationarity~\cite[proof of Lemma~4.1]{CF}; namely,
% RATES
the internal (true) and outflow reaction rates are: 
\begin{equation*}
r_{y \rightarrow y'} = \frac{\langle y, \delta \rangle}{e^{\langle y , \delta \rangle} - 1} \eta^0_{y \to y'}  \quad \quad {\rm for~all~} \  y \rightarrow y' \in \mathcal{R}_T\cup \mathcal{R}_O~,
\end{equation*}
% RATES
the inflow reaction rates are the coordinates of the following vector: 
\begin{equation} \label{eq:inflow}
\left( r_{0 \to X_i} \right) ~=~
\sum\limits_{y \to y'} \eta^0_{y \to y'}(y-y') ~ \in~ \mathbb{R}_+^{|\Spec|}
~,
\end{equation}
% STEADY STATES
and the two steady states are:
\begin{align*}
\mathbf{x^*} = (1, 1, ..., 1) \quad \quad {\rm and} \quad \quad 
\mathbf{x^\#} = (e^{\delta_1}, e^{\delta_2}, ..., e^{\delta_{|\mathcal{S}|}})~.
\end{align*}
%Now that all other rates and concentrations have been determined, the inflow rates for each species can be solved from that species' ODE in (\ref{M-A System}).
Craciun and Feinberg showed that for sufficiently small scaling of $\delta$, all inflow rates~\eqref{eq:inflow} are positive~\cite{CF}.

%In the case of $\widetilde{K}_{m,n}$, we use (\ref{ODEs}). 

\begin{example}
For $\widetilde{K}_{m,n}$, with $m \geq 2$ and $n=3,5,7,9,11$, we checked that $\delta_1=1$ suffices.  Details for the $n=3$ case are provided in Remark~\ref{rmk:inflow-pos-3}.
\end{example}

Summarizing what we accomplished above, we have closed-form expressions for reaction rate constants and steady states that show that the sequestration network is multistationary:
%
%The following result summarizes that has been accomplished above for the sequestration network; the main contribution is formulas for the reaction rate constants and steady states:
%
%--------------
% THEOREM: fully open system
%--------------
\begin{theorem} \label{thm:ratescon}
Consider positive integers $m \geq 2$ and $n\in\{3,5,9,11\}$.  
Let $\delta \in \mathbb{R}^n$ be as in (\ref{eq:delta}) with $\delta_1=1$.
Also, let $\eta^0$ be as is~\eqref{eq:eta-2n}\footnote{In fact, this theorem will hold for any larger $n$ for which the last coordinate of $\eta^0$ as defined in~\eqref{eq:eta-2n} can be made to be positive.}.  
Then, for the following internal (true) and outflow reaction rates: 
\begin{equation*} %\label{intoutrates2}
r_i = \frac{\langle y_i, \delta \rangle}{e^{\langle y_i , \delta \rangle} - 1} \eta_i^0 \ \ \ \ \ {\rm for~all~} \  i \in \{1, 2, ..., 2n\}~,
\end{equation*}
and the following
inflow reaction rates:
\begin{align*}
r_{2n+1}&=r_1+r_n+r_{n+1}\\ 
r_{2n+i}&=r_{i-1}+r_i+r_{n+i} \quad \quad {\rm for~all~} \text{ } 2\leq i \leq n-1\\
r_{3n}&=r_{n-1}+r_{2n}-mr_n~,
\end{align*}
the concentrations: 
\begin{align}
\mathbf{x^*} ~=~ (1, 1, ..., 1) \quad  \quad {\rm and}  %\label{concentration1}
\quad \quad \mathbf{x^\#} &= (e^{\delta_1}, e^{\delta_2}, ..., 
e^{\delta_n}) \label{concentration2}
\end{align}
both are positive steady states of the mass-action kinetics system defined by $\widetilde{K}_{m,n}$ and the reaction rates $r_i$ above.
\end{theorem}

%By the proof outlined in \cite{CF}, these will always satisfy the equation (\ref{M-A System}),
%$$\Gamma \cdot R(\mathbf{x^*}) = \Gamma \cdot R(\mathbf{x^\#}) = \mathbf{0}.$$

%\subsection{An example of Degeneracy using Determinant Optimization Method}
%\label{sec:degen}
\begin{remark} \label{rem:degen}
One may wonder whether or not the determinant optimization method has the potential to create degenerate steady states.  Indeed, if we could prove that this method always constructs nondegenerate steady states, then this would resolve Conjecture~\ref{Conjecture}.  However, this is not the case.

We determined this by analyzing $\widetilde{K}_{2,3}$ as in Theorem~\ref{thm:ratescon}. By letting $\epsilon$ be a free variable we compute the parametrized determinants $\det(df(\mathbf{x^{*}}))$ and $\det(df(\mathbf{x^{\#}}))$ as functions of $\epsilon$. Both functions are easily checked to be continuous for positive values of $\epsilon$,
and from the graph (Figure~\ref{remark311}), we can see easily that there exist choices of $\epsilon$ for which one of the two steady states is degenerate.  % one of the roots 
%and then we obtain ranges at which $\epsilon$ is a root of the determinant function. 
More precisely $\det(df(\mathbf{x^{*}}))=0$ for some $\epsilon \in (0.12,0.125)$ and $\det(df(\mathbf{x^{\#}}))=0$ for some $\epsilon \in (0.240,0.241)$ and some $\epsilon \in (1.159,1.160)$.

\begin{figure}[h]
\includegraphics[width=8cm]{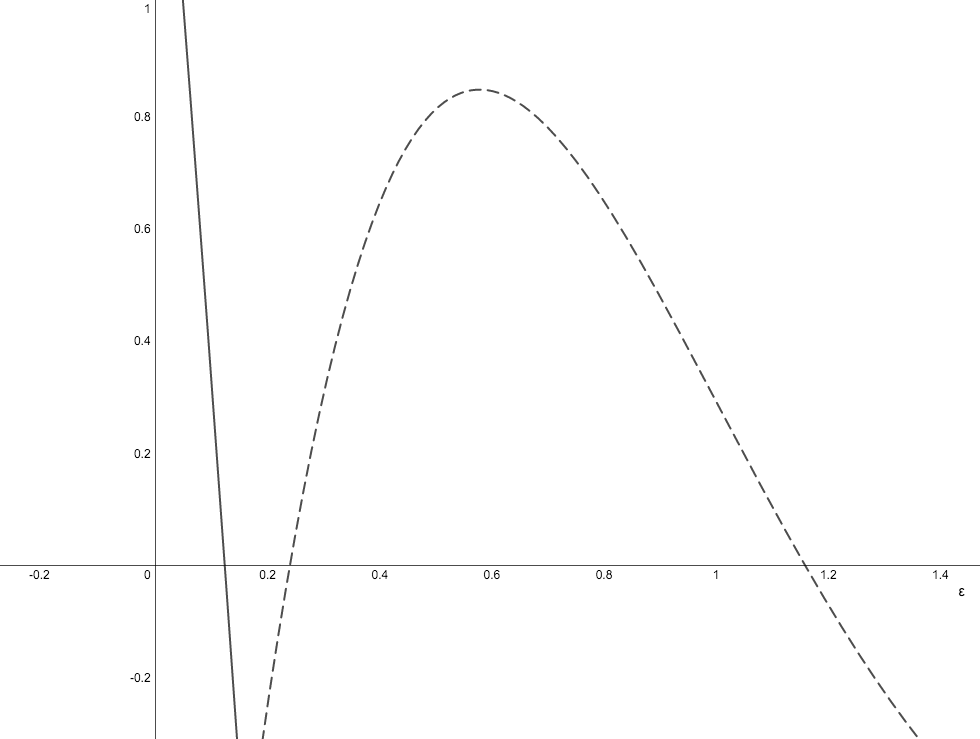}
\centering
\caption{Graphs of $\det(df(\mathbf{x^{*}}))$ {\em solid} and $\det(df(\mathbf{x^{\#}}))$ {\em dashed} as functions of $\epsilon$}.\label{remark311}
\end{figure}

\end{remark}

%---------------------
% SECTION -- CASE: n=3 (resolve the conjecture)
%---------------------
\section{Resolving Conjecture~\ref{Conjecture} for the $n=3$ case} \label{sec:n=3}
Recall that Conjecture~\ref{Conjecture} asserts that $\widetilde{K}_{m,n}$ admits multiple \emph{nondegenerate} positive steady states, for integers $m \geq 2$ and $n \geq 3$ with $n$ odd.
The main result of this section (Theorem~\ref{thm:resolve-conj}) resolves the conjecture when $n=3$. % and $m$ is any integer greater than or equal to 3.
%Here we discuss the case when we fix $n = 3$ and let $m$ be any integer $\geq 3$. 
To accomplish this, we first write down 
rate constants for this $n=3$ case for which there are two steady states $\mathbf{x^*}$ and $\mathbf{x^\#}$; 
these values were obtained by the determinant optimization method in the
previous section for the general $\widetilde{K}_{m,n}$ case
%to the $n=3$ case 
(Proposition~\ref{prop:n=3}).
We then resolve the conjecture for $n = 3$ by proving that $\mathbf{x^*}$ and $\mathbf{x^\#}$ are nondegenerate.

%---------------------
% subSECTION -- n=3 rates
%---------------------
\subsection{Reaction rate constants for which $\widetilde{K}_{m,3}$ is multistationary}
Proposition~\ref{prop:n=3} below specializes Theorem~\ref{thm:ratescon} to the $n=3$ case.  Following the description in Section~\ref{sec:generate}, $\lambda = 1$ and $\epsilon = 0.1$ will suffice, and then we obtain 
%We begin solving for our rates and concentrations using the method outlined in 
%First, we create a vector $\eta^- \in \mathbb{R}^{2n}_+$ such that $\sum_{i = 1}^{2n} \eta^-_i (y_i - y_i') \in \mathbb{R}_+^{2n}$. As mentioned in \cite{Shiu}, we simply let 
%$\eta^- = (\lambda, \lambda(m+1), \lambda, \epsilon, \epsilon, \epsilon)$.
%
%Graphical tests have shown for this case, $\lambda = 1$, $\epsilon = 0.1$ work sufficiently well. By Theorem 4.1 in \cite{CF} we know we can now solve for an $\eta^0$ such that the determinant of
%\begin{equation}
%T_\eta = 
%\begin{bmatrix}
%\eta_1 + \eta_3 + \eta_4 & \eta_1 & 0\\
%\eta_1 & \eta_1 + \eta_2 + \eta_5 & \eta_2 \\
%-m \eta_3 & \eta_2 & \eta_2 + \eta_6
%\end{bmatrix}
%\end{equation}
%is zero. One can simply let $\eta^0_1, ... \eta^0_{5}$ stay the same as $\eta^-$ and solve for $\eta_{6}$. Substituting and solving for $\eta_{6}$ we have:
%\begin{equation}
%\eta_{6} = 
%\frac{m^{2} - 0.31m - 1.31}{2.1m + 3.41}
%\end{equation}
%when $\lambda = 1$ and $\epsilon = 0.1$. 
    \begin{align*} %\label{eq:eta-0-n=3}
    \eta_0 ~=~ \left(   \lambda, ~\lambda(m+1), ~\lambda, ~\epsilon,~ \epsilon,~  \frac{m^{2} - 0.31m - 1.31}{2.1m + 3.41} \right)^T
    \end{align*}
from the second step of the determinant optimization method.  Next, in the third step, we find that the following vector spans the nullspace of $T_{\eta^0}$:
%$\delta$, we let $\delta_1 = 1$ and solve for the rest:
\begin{equation*} %\label{eq:delta-n=3}
\delta = 
%\begin{pmatrix}
%1 \\
%-2.1 \\
\left(1, ~ -2.1, ~ \frac{2.1m + 3.41}{m+1}~ \right)^T~.
%\end{pmatrix}
\end{equation*}
Thus, Theorem~\ref{thm:ratescon} specializes to:
%Armed with our vectors $\delta$ and $\eta^0$, we use (\ref{intoutrates}), (\ref{concentration1}), and (\ref{concentration2}) from Theorem \ref{thm:ratescon} to get our internal/inflow rates and concentrations: 

\begin{proposition} \label{prop:n=3}
Consider any integer $m \geq 2$, and the following internal, outflow, and internal reaction rates:
%   REACTION RATES
\begin{equation} \label{eq:rates-n=3}
\begin{matrix}
r_{1}=\frac{-1.1}{e^{-1.1}-1}\approx 1.65 &
r_{2}=\frac{1.31}{e^{\frac{1.31}{m+1}}-1} &
r_{3}=\frac{1}{e-1}\approx.58\\
r_{4}=\frac{.1}{e-1}\approx .06 &
r_{5}=\frac{-.21}{e^{-2.1}-1}\approx .24 &
r_{6}=\frac{m-1.31}{e^{\frac{2.1m+3.41}{m+1}}-1} \\
r_{7}= r_{1}+r_{3}+r_{4} \approx 2.29  &
r_{8}=r_{1}+r_{2}+r_{5} &
r_{9}=r_{2}+r_{6}-mr_{3}~.
\end{matrix}~
\end{equation}
Then for the mass-action kinetics system defined by the fully open sequestration network $\widetilde{K}_{m,3}$ and the above rate constants $r_{i}$, both 
$\mathbf{x^*}  =  (1, 1, 1) $ and $\mathbf{x^\#}  =  \left( e, e^{-2.1}, e^{\frac{2.1m+3.41}{m+1}} \right)$ are positive steady states.
\end{proposition}
\noindent
Note that in Proposition~\ref{prop:n=3}, only $\mathbf{x}^{\#}_3$, $r_{2}$,  $r_{6}$, $r_8$, and $r_9$ depend on $m$. 

\begin{remark} \label{rmk:inflow-pos-3}
The only reaction rate in~\eqref{eq:rates-n=3} that is not obviously positive is the inflow rate $r_9$, so we verify it here:
%that $r_9$ is positive for all $m \geq 2$:
\begin{align*}
r_{9}~&=~r_{2}+r_{6}-mr_{3}
~>~ r_2 +0 - m \left( \frac{1}{e-1} \right)
~\geq~ m-m \left( \frac{1}{e-1} \right) ~>~0~,
\end{align*}
where the second-to-last inequality follows from Lemma \ref{r2} below.
%It follows that all rates and concentrations for the case $n = 3$ are positive.
\end{remark}

%From this point on we make use of a slight abuse of notation, defining $ x_i = \mathbf{x^\#_i} $, since $\mathbf{x^*_i} = 1$. {\color{red} Let's try not to have to use this.}

%------
% Bounds on rates/ steady states
%-----
\subsection{Bounding rates and steady states of $\widetilde{K}_{m,3}$}
Here we give upper and lower bounds which we will use to prove that $\mathbf{x^*}$ and $\mathbf{x^\#}$ are nondegenerate.
The following bounds are on the third coordinate of $\mathbf{x^\#}$:
%{\em Remark} (Bounds on $\mathbf{x^\#_3}(m) = x_3(m)$) \newline
\begin{equation}
\label{x3}
e^\frac{2.1y + 3.41}{y + 1} ~\geq~ \mathbf{x}^\#_3 = e^\frac{2.1m + 3.41}{m + 1} ~>~ e^{2.1} \quad \quad {\rm for~all~}  m \ \geq \ y \  \geq 0~.
\end{equation}
The first inequality in~\eqref{x3} follows from the easy fact that $e^\frac{2.1m + 3.41}{m + 1}$ is a decreasing function when $m>0$, and the second inequality is straightforward.

The proofs of the following two upper/lower bounds are in Appendix~A:
%The proofs of the following Lemmas \ref{r2} and \ref{r6} are reserved for the Appendix A.

\begin{lemma}[Bounds on $r_2$]
\label{r2}
When $\lambda = 1$ and $\epsilon = 0.1$, the rate constant $r_2$ defined in~\eqref{eq:rates-n=3} satisfies the following inequalities for all $m \geq 2$:
\begin{align*}
m + 1 ~>~ r_2 ~\geq~ m ~.
%\quad \quad {\rm for~all~} m \geq 2~.
\end{align*}
\end{lemma}

\begin{lemma}[Bounds on $r_6$]
\label{r6}
When $\lambda = 1$ and $\epsilon = 0.1$, the rate constant $r_6$ defined in~\eqref{eq:rates-n=3} satisfies the following inequalities:
$$ 0.14m~ >~ r_6 ~> ~0.13m- 0.5~, $$
where the upper bound holds for $m \geq 2$, and the lower bound holds for $ m \geq 20$.
\end{lemma}

% -----
% Subsection -- proof of conjecture (n=3)
% -----
\subsection{Proving nondegeneracy of steady states for the network $\widetilde{K}_{m,3}$}

The main result of this section is:
%---------------------
% THEOREM (resolve the conjecture)
%---------------------
\begin{theorem}[Resolution of Conjecture \ref{Conjecture} when $n = 3$] \label{thm:resolve-conj}
For integers $m \geq 2$, the network $\widetilde{K}_{m,3}$ has the capacity to admit multiple nondegenerate positive steady states.
\end{theorem}

We will prove Theorem~\ref{thm:resolve-conj} by showing that $\mathbf{x^*}$ and $\mathbf{x^\#}$ in Proposition~\ref{prop:n=3} are nondegenerate, i.e.\  we must prove that the image of the $3 \times 3$ Jacobian matrix $df(\mathbf{x})$ at each of the steady states is equal to the image of the $3 \times 9$ matrix $\Gamma$. As stated earlier (after Conjecture~\ref{Conjecture}), since $\Gamma$ is full rank, our problem reduces to showing that $\det(df(\mathbf{x})) \neq 0 $ for both steady states and for all integers $m \geq 2$. 

We begin by displaying the Jacobian matrix~\eqref{eq:jac} of $\widetilde{K}_{m,3}$:
%for the case of $n = 3$:
%---------------------
% JACOBIAN MATRIX
%---------------------
\begin{equation*}
df(\mathbf{x}) = 
\begin{bmatrix}
-r_1x_2 - r_3 - r_{4} & -r_1x_1 & 0 \\
-r_1x_2 & -r_1x_1 - r_2x_3 - r_5 & -r_2x_2 \\
m r_3 & -r_2x_3 & -r_2x_2 - r_6
\end{bmatrix}~.
\end{equation*}
%where $x_i$ is the $i^{th}$ entry of the concentration vector $\mathbf{x}$.
Thus, %as we are only interested in whether the Jacobian matrix is full rank, 
our goal is to show that the following determinants (obtained by strategically cancelling and rearranging terms) are nonzero for all integers $m \geq 2$:
%---------------------
% THE DETERMINANTS
%---------------------
%Taking the Jacobian at both steady states, we have:
\begin{align}
\label{jac1}
\nonumber
D_1~&:=~ \det(df(\mathbf{x^{*}}))~=~
    r_{2}r_{1}r_{3}m ~-~ (r_{2}+r_{6})(r_{1}r_{3}+r_{1}r_{4}+r_{1}r_{5}+r_{3}r_{5}+r_{4}r_{5})  \\% -\\
    & \quad \quad \quad \quad \quad \quad \quad \quad \quad  \quad \quad ~-~ r_{2}r_{6}(r_{1}+r_{3}+r_{4}) \\
% DET 2
\label{jac2}
\nonumber
D_2~&:=~
\det(df(\mathbf{x^{\#}}))~=~r_{2}x^{\#}_{2}((r_{1}x^{\#}_{2}+r_{3} ~+~ r_{4})(r_{2}x^{\#}_{3})+ r_{1}x^{\#}_{1}mr_{3})  \\
& \quad \quad \quad \quad \quad \quad \quad \quad \quad \quad \quad-~ (r_{2}x^{\#}_{2}+ r_{6})(r_{1}x^{\#}_{2}+r_{3}+r_{4})(r_{1}x^{\#}_{1}+r_{2}x^{\#}_{3}+r_{5}) \nonumber \\
& \quad \quad \quad \quad \quad \quad \quad \quad \quad  \quad \quad +~ (r_{2}x^{\#}_{2}+r_{6})(r_{1}x^{\#}_{1}r_{1}x^{\#}_{2}) 
\end{align}

From its graph\footnote{Analogous graphs for larger $n$ appear in Appendix~B.}, $D_1$
%the first determinant (\ref{jac1})
%of the Jacobian at $x^{*}$ ), 
appears to be increasing quadratically as a function of $m$. So, to prove that 
%(\ref{jac1}) 
$D_1$ is nonzero for integer values of $m \geq 2$, we will bound it from below by a quadratic function. %This bound is obtained by relying on a series of inequalities. 
Similarly, $D_2$ % the determinant of %the second determinant at $x^{\#}$, 
%(\ref{jac2}) 
appears to be decreasing quadratically, so we will bound it from above by another quadratic function. Using these bounds, we will then conclude that $D_1$ and $D_2$
%(\ref{jac1}) and (\ref{jac2}) 
are strictly positive and negative (respectively) after certain cutoff points of $m$, effectively showing nondegeneracy of both steady states beyond the cutoffs. Finally, we will complete the proof by 
%, we offer numerical evidence to show 
evaluating $D_1$ and $D_2$
%(\ref{jac1}) and (\ref{jac2}) 
at the remaining integers $m$ between the $2$ and the cutoff points to show that these values are nonzero.

%PROOF
\begin{proof}[Proof of Theorem~\ref{thm:resolve-conj}]
We generate our rates and concentrations as in Proposition~\ref{prop:n=3}. 

Following the description immediately after Theorem~\ref{thm:resolve-conj}, we need only show that $D_1 \neq 0$ and $D_2 \neq 0$ for all integers $m \geq 2$.  
% BOUND FOR D_1
First, we bound $D_1$ by using its formula~\eqref{jac1} together with the bounds in %in~\eqref{x3} and 
Lemmas~\ref{r2} and~\ref{r6}:
\begin{align*}
D_1 
%\det(df(\mathbf{x^*}) 
\  &> \ m^2r_{1}r_{3} - ((m+1)+0.14m)(r_{1}r_{3}+r_{1}r_{4}+r_{1}r_{5}+r_{3}r_{5}+r_{4}r_{5}) \\
& \quad \quad -~ (m+1)(0.14 m)(r_{1}+r_{3}+r_{4})~.
\end{align*}
Next, estimating the remaining rates $r_i$, which are constants (recall equations~\eqref{eq:rates-n=3}), by appropriate upper or lower bounds, we obtain:
\begin{align*}
%\det(df(\mathbf{x^*}) 
D_1
&~>~ m^2(0.95) - ((m+1)+0.14m)1.61 -
(m+1)(0.14 m)2.29 \nonumber \\
&~=~ \ 0.6294m^2 - 2.156m - 1.61~.
\end{align*}
It is easy to show that the quadratic function which bounds ${D}_1$ above is always positive for integers $m >4$. 
So, $D_1>0$ for $m>4$.  Thus, it remains only to show that $D_1 \neq 0$ at $m=2,3,4$; indeed, those values are nonzero and are listed in Table~\ref{table:1}.

%-----------------
% TABLE
%-----------------
\begin{table}[ht]
\scalemath{.75}{
\centering
\begin{tabular}{||c || c c c c c c c c c  ||} 
 \hline
 $m$ & 2 & 3 & 4 & 5 & 6 & 7 & 8 & 9 & 10  \\ [0.5ex] 
 \hline
 $D_1~=~\det(df(\mathbf{x^*}))$ & 0.336 & 2.784  & 6.525 & & & & & &  \\ [1ex] 
 $D_2~=~\det(df(\mathbf{x^\#}))$ & -1.063 & -3.811 & -7.85 & -13.19 & -19.8 & -27.71 & -36.89 & -47.36 & -59.11  \\ 
 \hline\hline
 $m$ & 11 & 12 & 13 & 14 & 15 & 16 & 17 & 18 & 19  \\ [0.5ex] %& 20
 \hline
 $D_2~=~\det(df(\mathbf{x^\#}))$ & -72.14 & -86.4 & -102 & -118.9 & -137.1 & -156.5 & -177.2 & -199.2 & -222.5 \\%& -247.1 
 \hline
\end{tabular}
}
\caption{Determinants of the Jacobian matrices~(\ref{jac1}--\ref{jac2}) at the two steady states $\mathbf{x^*}$ and $\mathbf{x^\#}$ for the values of $m \geq 2$ before the proven bounds are valid. All of these determinants are nonzero, so the corresponding steady states are nondegenerate.}
\label{table:1}
\end{table}

% BOUND FOR D_2
Now we proceed to bound $D_2$. Again, we use its formula~\eqref{jac2} together with the bounds in~\eqref{x3} and Lemmas~\ref{r2} and~\ref{r6}:

\begin{align*}
D_2 %~=~
%\det(df(x^{\#}))
& ~<~ (m+1)x^{\#}_2((r_1x^{\#}_2+r_3+r_4)((m+1)x^{\#}_3)+r_1x^{\#}_1mr_3) \\
    & \quad \quad -~ (mx_2+(.13m-.5))(r_1x^{\#}_2+r_3+r_4)(r_1x^{\#}_1+mx^{\#}_3+r_5) \\
    & \quad \quad +~ ((m+1)x^{\#}_2+.14m)(r_1x^{\#}_1r_1x^{\#}_2)~.
\end{align*}
Note that we used the lower bound on $r_6$, so the above inequality holds for $m \geq 20$ (and thus we will need to check the values of $m$ between 2 and 19 separately).
%we will need to numerically check all values of the second determinant (\ref{jac2}) for $2 \leq m \leq 20$. \newline  
In the same manner as before, we approximate all of the constants appropriately for $m \geq 20$, and then simplify:
\begin{align*}
%\det(df(x^{\#})) 
D_2
~&<~
(m+1).13((.85)((m+1)8.7)+2.61m)- (.25m-.5)(.84)(4.72+8.16m)\\
& \quad \quad \quad \quad +~ (.13(m+1)+.14m)(.91) \\
%\det(df(x^{\#}))<-\\
& = -0.41295m^2+4.9437m+3.06205~. 
\end{align*}
Therefore, it is easy to see that $D_2$ is nonzero for $m\geq 20$. For $2 \leq m \leq 19$ we again refer to Table~\ref{table:1}, which completes our proof.

\end{proof}

%-----------------
% SECTION -- NONDEGENERACY FOR SMALL M, n
%-----------------
\section{Resolving Conjecture~\ref{Conjecture} for small values of $m$ and $n$} \label{sec:small-val}
The main result of this section extends Theorem~\ref{thm:resolve-conj} to $n \leq 11$, for small $m$:

\begin{theorem}[Resolution of Conjecture \ref{Conjecture} for small $m$ and $n$] \label{thm:small-m-n}
For $m=2,3,4,5$ and $n=5,7,9,11$, the network $\widetilde{K}_{m,n}$ has the capacity to admit multiple nondegenerate positive steady states.
\end{theorem}

\begin{proof}
We generate our rates and concentrations as in Theorem~\ref{thm:ratescon}: it is straightforward to check that $\delta_1 = 1$, $\lambda = 1$, and $\epsilon = 0.001$ satisfy all necessary hypotheses.  Thus, we obtain two steady states, $x^*= (1,1,\dots,1)$ and $x^\#$ defined in (\ref{concentration2}).  Then, as in the proof of Theorem~\ref{thm:ratescon}, we verify that the determinants $\det(df(x^{*}))$ and $\det(df(x^{\#}))$ are nonzero for $m=2,3,4,5$, which is readily seen from their graphs, which appear in Appendix B.  (In fact, the graphs strongly suggest that the conjecture holds completely for each of these values of $n$, namely $n=5,7,9,11$, i.e.\ for $m >5$ as well.)
\end{proof}

%------------
% DISCUSSION
%------------
\section{Discussion} \label{sec:disc}
As stated in the introduction, deciding whether a chemical reaction network is multistationary is not easy in the general case. And even when we can confirm that a network is multistationary, there is no general technique to show that it will admit multiple {\em nondegenerate} steady states. Nonetheless, in this paper we succeeded in this task for certain sequestration networks $\widetilde{K}_{m,n}$ by using the determininant optimization method to obtain closed forms for reaction rates and steady states. 

Our work resolved the $n=3$ case of Conjecture~\ref{Conjecture}, and we believe that our results form an important step toward resolving the full conjecture.  Specifically, one could use the formulas for rates and steady states given in Theorem~\ref{thm:ratescon} to analyze the general case, or, perhaps easier, the case of some fixed $m$ and general $n$.
Two other possible approaches are to (1)
find an alternate method to obtain closed forms for the steady states of a chemical reaction network, or
(2) identify criteria that can guarantee that steady states are nondegenerate. 

Expanding on the last idea, our ultimate goal is to develop general techniques to assert that steady states of a chemical reaction network are nondegenerate. For instance, our analysis of the Jacobian determinants in this work suggest that even if the determinant optimization method yields a degenerate steady state, then the rate constants can be perturbed slightly so that the degenerate steady state becomes nondegenerate (and the other steady state also remains nondegenerate).  Is this true for any network for which the determinant optimization method applies?  If so, then this would completely resolve Conjecture~\ref{Conjecture}, and, more generally, this would enable us to more readily ``lift'' multistationarity and thereby enlarge our catalogue of known multistationary networks.

%*******************************************************************
%Acknowledgements
%*******************************************************************
\subsection*{Acknowledgements}
\noindent
BF and ZW conducted this research as part of the NSF-funded REU in the Department of Mathematics at Texas A\&M University (DMS-1460766), in which AS served as mentor.  All authors contributed substantially to this work.  
AS was supported by the NSF (DMS-1312473).  The authors thank Dean Baskin for help with the proof of Lemma~\ref{r2}, and Maya Johnson, Badal Joshi, Emma Owusu
Kwaakwah, Casian Pantea, Xiaoxian Tang, and Jacob White for advice and fruitful discussions.  The authors also thank an anonymous referee.

%\section*{References}

\bibliography{mybibfile}

%------------
% APPENDIX A -- PROOFS OF LEMMAS
%------------
\section*{Appendix A: proofs of Lemmas \ref{r2} and \ref{r6}.}
%Here we give the proofs of Lemmas \ref{r2} and \ref{r6}.

\begin{lemma}[Lemma \ref{r2}]% (Bounds on $r_2$).\textit{ 
The function $r_2(m)= \frac{1.31}{e^{\frac{1.31}{m+1}}-1} $ satisfies the following inequalities for all $m \geq 2$:
$$m + 1 ~>~ r_2(m) ~\geq~ m ~.
$$
\end{lemma}

\begin{proof}
% Recall $r_2 =  \frac{1.31}{e^{\frac{1.31}{m+1}}-1}$. T
For the upper bound, first observe that 
$$\log \left( \left(1 + \frac{1.31}{m+1} \right)^{m+1} \right) ~<~ \log(e^{1.31}) ~=~ 1.31 \quad \quad {\rm for~all~} \ m \geq 0~,$$
since $\lim_{x \rightarrow \infty} (1 + \frac{y}{x})^x$ converges to $e^y$ from below for positive values of $y$. Thus:
\begin{align*}
1.31 ~>~ & \ \log \left( \left( 1 + \frac{1.31}{m+1} \right)^{m+1} \right)  \\
 ~>~ & \ \log \left( \left( \frac{m+2.31}{m+1} \right) ^{m+1} \right)
 ~=~(m+1) \log \left( \frac{m+2.31}{m+1} \right)~,
 \end{align*}
which implies that
\begin{align*}
\frac{1.31}{m+1} & ~>~ \log \left( \frac{m+2.31}{m+1} \right) \\
\implies e^\frac{1.31}{m+1} & ~>~ \frac{m+2.31}{m+1} \\
\implies e^\frac{1.31}{m+1}(m+1) - (m + 1) &~>~ 1.31~. %\ \ \ \ \ {\rm for~all~} m \geq 2.
\end{align*}
and this final inequality implies our desired upper bound:  $m + 1 ~>~ \frac{1.31}{e^\frac{1.31}{m+1} - 1} = r_2(m)$.

Notice that our lower bound is equivalent
% Notice that our claim, $\frac{1.31}{e^{\frac{1.31}{m+1}} - 1} > m$, is equivalent 
to the following inequality:
\begin{align} \label{eq:lower-bd-rewritten}
\frac{m + 1.31}{m} ~\geq~ e^{\frac{1.31}{m+1}} \quad \quad {\rm for~all~} m \geq 2~.
\end{align}
We set $a = 1.31$, make use of the change of variables $z = \frac{1}{m}$, and then 
%Note that, because of this change of variables, we only consider $0 < z \leq \frac{1}{2}$. 
apply $\log$ to see that our desired inequality~\eqref{eq:lower-bd-rewritten} is equivalent to:
\begin{align*}%\label{eq:lower-bd-2}
\log(1 + az)~ \geq ~ \frac{a}{z^{-1} + 1} ~=~ \frac{az}{1 + z} \quad \quad {\rm for~all~} z \in (0,1/2)~.
\end{align*}
We will show that $\log(1 + az) -\frac{az}{1 + z} \geq 0$. 
To this end, define $b$ by $1 - b = a - 1$, and notice that $1 > b > \frac{1}{2} > (a-1)$. 
Next, note that we have the following equalities:
{\tiny
\begin{align} \label{eq:integrals}
\log(1 + az) -  \frac{az}{1 + z} \nonumber ~&=~ \int_0^{az} \left( \frac{1}{1+t} - \frac{1}{1+z} \right) dt
\\
~&=~ \int_0^{bz} \frac{z - t}{(1+z)(1+t)} dt + \int_{bz}^z \left( \frac{1}{1+t} - \frac{1}{1+z} \right)  + \int_z^{az} \frac{z - t}{(1+z)(1+t)} dt~.
\end{align}
}%
The second integral in~\eqref{eq:integrals} is nonnegative (because its integrand is nonnegative), so we complete the proof now by showing that the sum of the first and third integrals in~\eqref{eq:integrals} is nonnegative:
\begin{align*}
\int_0^{bz} \frac{z - t}{(1+z)(1+t)} dt + 
    \int_z^{az} \frac{z - t}{(1+z)(1+t)} dt 
~\geq~ \frac{(1-b)z^2b}{(z+1)^2} +     
    \frac{-(a-1)^2z^2}{(z+1)^2}  ~\geq~0~,
\end{align*}
where the two inequalities come from recalling that $b <1$, and, respectively, $(1 - b) = (a - 1)$ and $b \geq (a - 1)$.
%, we see $\frac{(1-b)z^2b}{(z+1)^2}$ outweighs $\frac{-(a-1)^2z^2}{(z+1)^2}$, thus we have
%$$\log(1 + az) -\frac{az}{1 + z} = \int_0^{az} \left( \frac{1}{1+t} - \frac{1}{1+z} \right) dt \geq 0 $$ 
\end{proof}

\begin{lemma}[Lemma \ref{r6}] %(Bounds on $r_6(m)$)
%When $\lambda = 1$, and $\epsilon = 0.1$, 
The function $r_6(m)=\frac{m - 1.31}{e^{\frac{2.1m + 3.41}{m+1}} - 1}$ satisfies:
\begin{align} \label{eq:ineq-2}
0.14m ~>~ r_6(m) ~>~ 0.13m- 0.5~,
\end{align}
where the upper bound holds for $m \geq 2$, and the lower bound holds for $ m \geq 20$.
\end{lemma}

\begin{proof}
%Recall $r_6 = \frac{m - 1.31}{e^{\frac{2.1m + 3.41}{m+1}} - 1}$. 
We first prove the upper bound.  
By the second inequality in (\ref{x3}),  $(e^{\frac{2.1m + 3.41}{m+1}} - 1) > 0$ for $m \geq 2$. Thus our desired upper bound in~\eqref{eq:ineq-2} is equivalent to the following:
$$m(0.14e^{\frac{2.1m + 3.41}{m+1}} - 1.14)~>~ -1.31 ~,$$
which holds (for positive $m$) whenever $(0.14e^{\frac{2.1m + 3.41}{m+1}} - 1.14) > 0$. 
This inequality in turn is equivalent to the following (since $\log$ is an increasing function):
%, we can algebraically manipulate the expression $0.14e^{\frac{2.1m + 3.41}{m+1}} > 1.14$ to see
$$2.1m + 3.41 ~>~ (m+1)\log\left(\frac{1.14}{0.14}\right) ~\approx~ (m+1)(2.10)~,$$
which is true for positive $m$, so the proof of the upper bound is complete.
%which is true for all $m$ satisfying $.01m > -1.32 \implies m \geq -132$. So this inequality is true for all $m \geq 2.$ 

For the lower bound, % in~\eqref{eq:ineq-2}, 
by clearing the denominator and gathering exponential terms on the right-hand side, we see that the desired inequality is equivalent to the following:
$$1.13m - 1.81 ~\geq~ e^{\frac{2.1m + 3.41}{m+1}} (0.13m - 0.5)~.$$
We prove this now.  The first inequality below is equivalent to the inequality $.00004m \geq -2.536$, which is true for positive $m$:
\begin{align*}
1.13m - 1.81 ~&\geq~
8.692(0.13m - 0.5) \\
~&>~ e^{\frac{2.1(20) + 3.41}{(20)+1}}(0.13m - 0.5) 
~\geq~
e^{\frac{2.1m + 3.41}{m+1}} (0.13m - 0.5)~,
\end{align*}
and the final inequality holds for $m \geq 20$ because $e^{\frac{2.1m + 3.41}{m+1}}$ is a decreasing function for positive values of $m$.
\end{proof}

\section*{Appendix B: graphs for the proof of Thoeorem~\ref{thm:small-m-n}}
Figures~\ref{fig:sfig1}--\ref{fig:sfig4} below present the graphs of the determinant of the Jacobian matrix evaluated at the steady states $x^*$ and $x^{\#}$ (as described in the proof of Theorem~\ref{thm:small-m-n}) for the network $\widetilde{K}_{m,n}$ for odd $5 \leq n \leq 11$ as functions of $m$. 
%The steady states ($r$, $x^*$ and $x^{\#}$) are obtained using the method described in Section 3. 
Note that the graphs are nonzero for $2 \leq m \leq 5$, confirming Conjecture~\ref{Conjecture} for odd $5 \leq n \leq 11$ and those values of $m$.  Also, the graphs strongly suggest that the conjecture holds for larger $m$ as well.  All graphs were made using {\tt Desmos Graphing Calculator}~\cite{desmos}.

\begin{figure}[ht]
\begin{subfigure}{.5\textwidth}
  \centering
  \includegraphics[width=.8\linewidth]{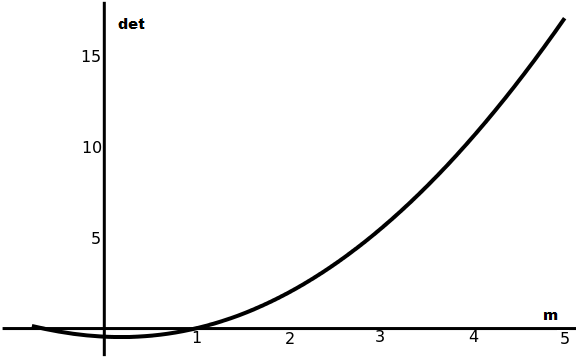}
  \caption{$\det(df(\mathbf{x^{*}}))$}
  \label{fig:sfig1}
\end{subfigure}
\begin{subfigure}{.5\textwidth}
  \centering
  \includegraphics[width=.8\linewidth]{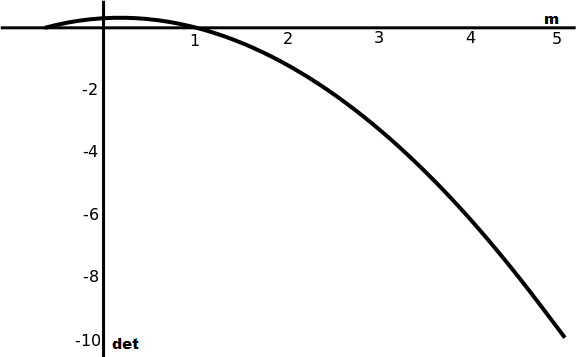}
  \caption{$\det(df(\mathbf{x^{\#}}))$}
  \label{fig:sfig2}
\end{subfigure}
\caption {$\widetilde{K}_{m,5}$}
\label{fig:fig}
\end{figure}
\begin{figure}[ht]
\begin{subfigure}{.5\textwidth}
  \centering
  \includegraphics[width=.8\linewidth]{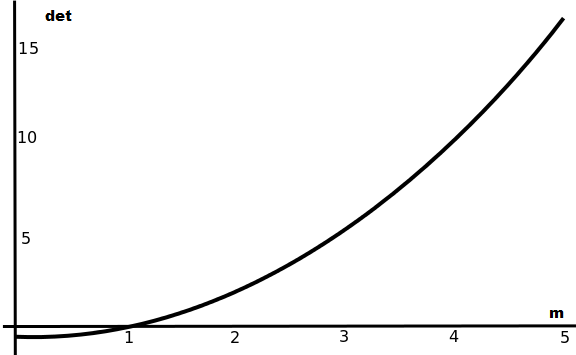}
  \caption{$\det(df(\mathbf{x^{*}}))$}
  \label{fig:sfig4}
\end{subfigure}
\begin{subfigure}{.5\textwidth}
  \centering
  \includegraphics[width=.8\linewidth]{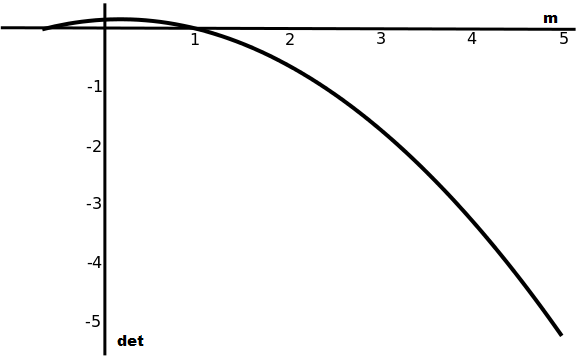}
  \caption{$\det(df(\mathbf{x^{\#}}))$}
  \label{fig:sfig3}
\end{subfigure}
\caption{$\widetilde{K}_{m,7}$}
\label{fig:fig2}
\end{figure}
\begin{figure}[ht]
\begin{subfigure}{.5\textwidth}
  \centering
  \includegraphics[width=.8\linewidth]{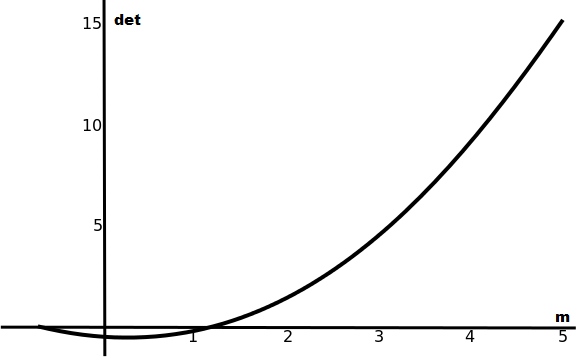}
  \caption{$\det(df(\mathbf{x^{*}}))$}
  \label{fig:sfig6}
\end{subfigure}
\begin{subfigure}{.5\textwidth}
  \centering
  \includegraphics[width=.8\linewidth]{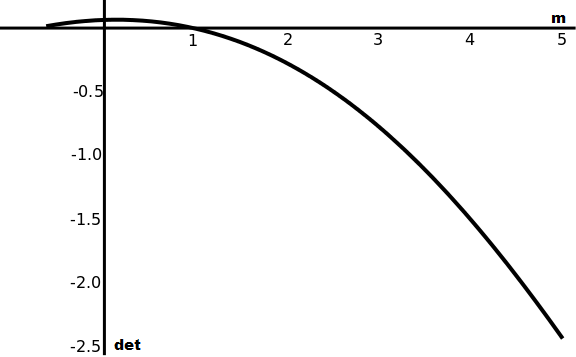}
  \caption{$\det(df(\mathbf{x^{\#}}))$}
  \label{fig:sfig5}
\end{subfigure}
\caption{$\widetilde{K}_{m,9}$}
\label{fig:fig3}
\end{figure}
\begin{figure}[ht]
\begin{subfigure}{.5\textwidth}
  \centering
  \includegraphics[width=.8\linewidth]{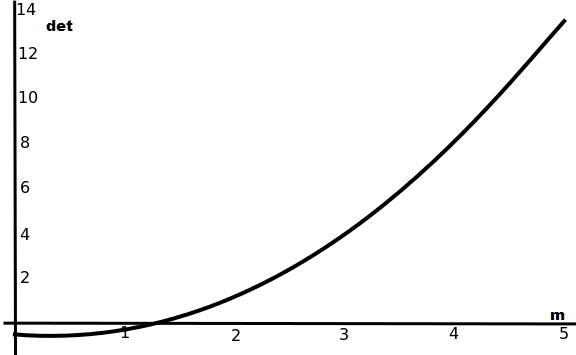}
  \caption{$\det(df(\mathbf{x^{*}}))$}
 \label{fig:sfig8}
\end{subfigure}
\begin{subfigure}{.5\textwidth}
  \centering
  \includegraphics[width=.8\linewidth]{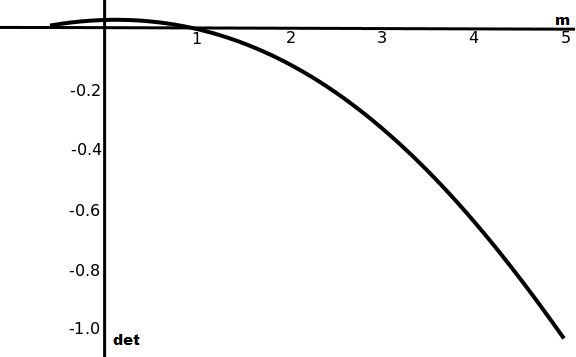}
  \caption{$\det(df(\mathbf{x^{\#}}))$}
  \label{fig:sfig7}
\end{subfigure}
\caption{$\widetilde{K}_{m,11}$}
\label{fig:fig4}
\end{figure}
%\newpage{}

\end{document}